\documentclass{article}

\usepackage{graphicx}      
\usepackage[justification=centering]{caption}
\usepackage[caption=false,font=footnotesize]{subfig}
\usepackage{amsmath} 
\usepackage{algorithm2e}
\usepackage[table]{xcolor}
\usepackage{amssymb} 
\usepackage{amsthm}

\newtheorem{proposition}{Proposition}

\begin{document}

\title{Optimal matching between curves in a manifold}

\author{Alice Le Brigant \and Marc Arnaudon \and Fr\'ed\'eric Barbaresco} 
\date{}

\maketitle              

\begin{abstract}
This paper is concerned with the computation of an optimal matching between two manifold-valued curves. Curves are seen as elements of an infinite-dimensional manifold and compared using a Riemannian metric that is invariant under the action of the reparameterization group. This group induces a quotient structure classically interpreted as the "shape space". We introduce a simple algorithm allowing to compute geodesics of the quotient shape space using a canonical decomposition of a path in the associated principal bundle. We consider the particular case of elastic metrics and show simulations for open curves in the plane, the hyperbolic plane and the sphere.
\end{abstract}

\section{Introduction}

A popular way to compare shapes of curves is through a Riemannian framework. The set of curves is seen as an infinite-dimensional manifold on which acts the group of reparameterizations, and is equipped with a Riemannian metric $G$ that is invariant with respect to the action of that group. Here we consider the set of open oriented curves in a Riemannian manifold $\left(M,\langle \cdot,\cdot \rangle\right)$ with velocity that never vanishes, i.e. smooth immersions,
\begin{equation*}
\mathcal{M}=\text{Imm}([0,1],M) = \{ c\in C^\infty([0,1],M) : c'(t) \neq 0 \,\, \forall t\in [0,1] \}.
\end{equation*}
It is an open submanifold of the Fr\'echet manifold $C^\infty([0,1],M)$ and its tangent space at a point $c$ is the set of infinitesimal vector fields along the curve $c$ in $M$,
\begin{equation*}
T_c\mathcal M = \{ w \in C^\infty([0,1],TM) : w(t) \in T_{c(t)}M \,\,\forall t\in [0,1] \}.
\end{equation*}
A curve $c$ can be reparametrized by right composition $c\circ \varphi$ with an increasing diffeomorphism $\varphi : [0,1] \rightarrow [0,1]$, the set of which is denoted by $\text{Diff}^+([0,1])$. We consider the quotient space $\mathcal S = \mathcal M/\text{Diff}^+([0,1],M)$, interpreted as the space of "shapes" or "unparameterized curves". 
If we restrict ourselves to elements of $\mathcal M$ on which the diffeomorphism group acts freely, then we obtain a principal bundle $\pi : \mathcal M \rightarrow \mathcal S$, the fibers of which are the sets of all the curves that are identical modulo reparameterization, i.e. that project on the same "shape" (Figure \ref{fibre}). We denote by $\bar c:=\pi(c)\in \mathcal S$ the shape of a curve $c\in \mathcal M$. Any tangent vector $w\in T_c\mathcal M$ can then be decomposed as the sum of a vertical part $w^{ver}\in \text{Ver}_c$, that has an action of reparameterizing the curve without changing its shape, and a horizontal part $w^{hor}\in\text{Hor}_c = \left(\text{Ver}_c \right)^{\perp_G}$, $G$-orthogonal to the fiber,
\begin{gather*}
T_c\mathcal M \ni w = w^{ver} + w^{hor} \in \text{Ver}_c \oplus \text{Hor}_c,\\
\text{Ver}_c = \ker T_c\pi = \left\{ m v := mc'/|c'| : \,\, m \in C^\infty([0,1], \mathbb R), m(0)=m(1)=0 \right\},\\
\text{Hor}_c = \left\{ h\in T_c\mathcal M : \,\, G_c(h,mv)=0, \,\, \forall m\in C^\infty([0,1],\mathbb R), m(0)=m(1)=0 \right\}.
\end{gather*}
If we equip $\mathcal M$ with a Riemannian metric $G_c : T_c\mathcal M \times T_c\mathcal M  \rightarrow \mathbb R$, $c\in \mathcal M$, that is constant along the fibers, i.e. such that
\begin{equation}
\label{equivariance}
G_{c\circ \varphi}(w\circ \varphi, z\circ \varphi) = G_c(w,z), \quad \forall \varphi \in \text{Diff}^+([0,1]),
\end{equation}
then there exists a Riemannian metric $\bar G$ on the shape space $\mathcal{S}$ such that $\pi$ is a Riemannian submersion from $(\mathcal M,G)$ to $(\mathcal S,\bar G)$, i.e.
\begin{equation*}
G_c(w^{hor},z^{hor}) = \bar G_{\pi( c)}\left( T_c\pi(w), T_c\pi(z) \right), \quad \forall w, z\in T_c\mathcal M.
\end{equation*}
This expression defines $\bar G$ in the sense that it does not depend on the choice of the representatives $c$, $w$ and $z$ (\cite{mich4}, \S 29.21). If a geodesic for $G$ has a horizontal initial speed, then its speed vector stays horizontal at all times - we say it is a horizontal geodesic - and projects on a geodesic of the shape space for $\bar G$ (\cite{mich4}, \S 26.12). The distance between two shapes for $\bar G$ is given by
\begin{equation*}
\bar d\left( \overline{c_0} , \overline{c_1} \right) = \inf \left\{\, d\left(c_0, c_1\circ \varphi \right) \, | \, \, \varphi \in \text{Diff}^+([0,1]) \, \right\}.
\end{equation*}
Solving the boundary value problem in the shape space can therefore be achieved either through the construction of horizontal geodesics e.g. by minimizing the horizontal path energy \cite{bauer12},\cite{tum}, or by incorporating the optimal reparameterization of one of the boundary curves as a parameter in the optimization problem \cite{bauer17},\cite{sri},\cite{zhang16}. Here we introduce a simple algorithm that computes the horizontal geodesic linking an initial curve with fixed parameterization $c_0$ to the closest reparameterization $c_1\circ\varphi$ of the target curve $c_1$. The optimal reparameterization $\varphi$ yields what we will call an \emph{optimal matching} between the curves $c_0$ and $c_1$.

\begin{figure}
\centering
\includegraphics[width=18em]{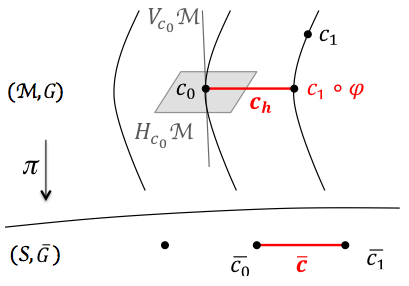}
\caption{Schematic representation of the shape bundle.}
\label{fibre}
\end{figure}

\section{The optimal matching algorithm}

We want to compute the geodesic path $s \mapsto \bar c(s)$ between the shapes of two curves $c_0$ and $c_1$, that is the projection $\bar c = \pi(c_h)$ of the horizontal geodesic $s \mapsto c_h(s)$ - if it exists - linking $c_0$ to the fiber of $c_1$ in $\mathcal M$, see Figure \ref{fibre}. This horizontal path verifies $c_h(0)=c_0$, $c_h(1) \in\pi^{-1}(\overline{c_1})$ and $\partial c_h/\partial s(s) \in \text{Hor}_{c_h(s)}$ for all $s\in[0,1]$. Its end point gives the optimal reparameterization $c_1\circ \varphi$ of the target curve $c_1$ with respect to the initial curve $c_0$, i.e. such that
\begin{equation*}
\bar d(\overline{c_0},\overline{c_1}) = d(c_0,c_1\circ \varphi) = d(c_0, c_h(1)).
\end{equation*}
In all that follows we identify a path of curves $[0,1]\ni s\mapsto c(s)\in \mathcal M$ with the function of two variables $[0,1]\times [0,1]\ni(s,t)\mapsto c(s,t)\in M$ and denote by $c_s:=\partial c/\partial s\,$ and $c_t:=\partial c/\partial t\,$ its partial derivatives with respect to $s$ and $t$. 
We decompose any path of curves $s\mapsto c(s)$ in $\mathcal M$ into a horizontal path reparameterized by a path of diffeomorphisms, i.e. $c(s) = c^{hor}(s) \circ \varphi(s)$ where $c^{hor}_s(s)\in \text{Hor}_{c^{hor}(s)}$ and $\varphi(s) \in \text{Diff}^+([0,1])$ for all $s\in[0,1]$. That is,
\begin{equation}
\label{def}
c(s,t) = c^{hor}(s,\varphi(s,t)) \quad \forall s,t\in[0,1].
\end{equation}
The horizontal and vertical parts of the speed vector of $c$ can be expressed in terms of this decomposition. Indeed, by taking the derivative of \eqref{def} with respect to $s$ and $t$ we obtain
\begin{subequations}
\begin{align}
c_s(s) &= c^{hor}_s(s) \circ \varphi(s) + \varphi_s(s) \cdot c^{hor}_t(s)\circ \varphi(s), \label{cs}\\
c_t(s) &= \varphi_t(s) \cdot c^{hor}_t(s) \circ \varphi(s), \label{ct}
\end{align}
\end{subequations}
and so if $v^{hor}(s,t) := c^{hor}_t(s,t)/|c^{hor}_t(s,t)|$ denotes the normalized speed vector of $c^{hor}$, \eqref{ct} gives since $\varphi_t >0$, $v(s) = v^{hor}(s) \circ \varphi(s)$. 
We can see that the first term on the right-hand side of Equation \eqref{cs} is horizontal. Indeed, for any $m : [0,1] \rightarrow C^\infty([0,1],\mathbb R)$ such that $m(s,0)=m(s,1)=0$ for all $s$, since $G$ is reparameterization invariant we have
\begin{align*}
G\left(c^{hor}_s(s) \circ \varphi(s), \,m(s) \cdot v(s) \right) &= G\left(c^{hor}_s(s)\circ \varphi(s), \,m(s) \cdot v^{hor}(s)\circ \varphi(s) \right)\\
&= G\left(c^{hor}_s(s), \,m(s) \circ \varphi(s)^{-1} \cdot v^{hor}(s) \right)\\
&= G\left(c^{hor}_s(s), \, \tilde m(s)\cdot v^{hor}(s)\right),
\end{align*}
with $\tilde m(s) = m(s) \circ \varphi(s)^{-1}$. Since $\tilde m(s,0) = \tilde m(s,1) = 0$ for all $s$, the vector $\tilde m(s)\cdot v^{hor}(s)$ is vertical and its scalar product with the horizontal vector $c^{hor}_s(s)$ vanishes. On the other hand, the second term on the right hand-side of Equation \eqref{cs} is vertical, since it can be written
\begin{equation*}
\varphi_s(s)\cdot c^{hor}_t\circ \varphi(s) = m(s)\cdot v(s),
\end{equation*}
with $m(s)=|c_t(s)|\varphi_s(s)/\varphi_t(s)$ verifying $m(s,0)=m(s,1)=0$ for all $s$. Finally, the vertical and horizontal parts of the speed vector $c_s(s)$ are given by
\begin{subequations}
\begin{align}
c_s(s)^{ver} &= m(s) \cdot v(s) = |c_t(s)|\varphi_s(s)/\varphi_t(s) \cdot v(s), \label{csver}\\
c_s(s)^{hor} &= c_s(s) - m(s)\cdot v(s) = c^{hor}_s(s)\circ \varphi(s).\label{cshor}
\end{align}
\end{subequations}
We call $c^{hor}$ the \emph{horizontal part} of the path $c$ with respect to $G$.
\begin{proposition}
The horizontal part of a path of curves $c$ is at most the same length as $c$  
\begin{equation*} L_G(c^{hor}) \leq L_G(c).
\end{equation*}
\end{proposition}
\begin{proof}
Since the metric $G$ is reparameterization invariant, the squared norm of the speed vector of the path $c$ at time $s\in[0,1]$ is given by, if $\|\cdot\|^2_G:=G(\cdot,\cdot)$,
\begin{align*}
\|c_s(s,\cdot)\|_G^2 &= \|c^{hor}_s(s,\varphi(s,\cdot))\|_G^2 + |\varphi_s(s,\cdot)|^2\|c^{hor}_t(s,\varphi(s,\cdot)\|_G^2\\
&=\|c^{hor}_s(s,\cdot)\|_G^2 + |\varphi_s(s,\cdot)|^2 \| c^{hor}_t(s,\cdot)\|_G^2,
\end{align*}
This gives $\|c^{hor}_s(s)\|_G\leq \|c_s(s)\|$ for all $s$ and so $L_G(c^{hor})\leq L_G(c)$.
\end{proof}
Now we will see how the horizontal part of a path of curves can be computed.
\begin{proposition}[Horizontal part of a path]\label{horpath}
Let $s\mapsto c(s)$ be a path in $\mathcal M$. Then its horizontal part is given by $c^{hor}(s,t) = c(s,\varphi(s)^{-1}(t))$, where the path of diffeomorphisms $s\mapsto \varphi(s)$ is solution of the PDE
\begin{equation}
\label{phicont}
\varphi_s(s,t) = m(s,t)/|c_t(s,t)|\cdot \varphi_t(s,t),
\end{equation}
with initial condition $\varphi(0,\cdot)=\text{Id}$, and where $m(s) : [0,1] \rightarrow \mathbb R$, $t \mapsto m(s,t):=|c_s^{ver}(s,t)|$ is the vertical component of $c_s(s)$.
\end{proposition}
\begin{proof}
This is a direct consequence of Equation \eqref{csver}, which states that the vertical part of $c_s(s)$ is $m(s)\cdot v(s)$ where $m(s) = |c_t(s)|\varphi_s(s)/\varphi_t(s)$.
\end{proof}
If we take the horizontal part of the geodesic linking two curves $c_0$ and $c_1$, we will obtain a horizontal path linking $c_0$ to the fiber of $c_1$ which will no longer be a geodesic path. However this path reduces the distance between $c_0$ and the fiber of $c_1$, and gives a "better" representative $\tilde c_1 = c_1\circ \varphi(1)$ of the target curve. By computing the geodesic between $c_0$ and this new representative $\tilde c_1$, we are guaranteed to reduce once more the distance to the fiber. The algorithm that we propose simply iterates these two steps and is detailed in Algorithm \ref{alg}.

\begin{algorithm}[h]
\KwData{$c_0,c_1\in \mathcal M$}
\KwResult{$\tilde c_1$}
Set $\tilde c_1 \leftarrow c_1$ and $\text{Gap}\leftarrow 2\times\text{Threshold}$\;
\While{$\text{Gap}>\text{Threshold}$}{
construct the geodesic $s\mapsto c(s)$ between $c_0$ and $\tilde c_1$\; 
compute the horizontal part $s\mapsto c^{hor}(s)$ of $c$\;
set $\text{Gap}\leftarrow \text{dist}_{L^2}\big(c^{hor}(1),\tilde c_1\big)$ and $\tilde c_1 \leftarrow c^{hor}(1)$\; }
\caption{Optimal matching.}
\label{alg}
\end{algorithm}

\section{Example : elastic metrics}

In this section we consider the particular case of the two-parameter family of elastic metrics, introduced for plane curves by Mio et al. in \cite{mio}. 
We denote by $\nabla$ the Levi-Civita connection of the Riemannian manifold $M$, and by $\nabla_tw := \nabla_{c_t}w$, $\nabla^2_tw := \nabla_{c_t}\nabla_{c_t}w$ the first and second order covariant derivatives of a vector field $w$ along a curve $c$ of parameter $t$. For manifold-valued curves, elastic metrics can be defined for any $c\in T_c\mathcal M$ and $w,z\in T_c\mathcal M$ by
\begin{equation}
\label{metric}
G^{a,b}_c(w,z) = \langle w(0), z(0) \rangle + \int_0^1 \left(a^2\langle \nabla_\ell w^N, \nabla_\ell z^N \rangle + b^2 \langle \nabla_\ell w^T, \nabla_\ell z^T\rangle \right) \mathrm d\ell,
\end{equation}
where $\mathrm d\ell = |c'(t)|\mathrm dt$ and $\nabla_\ell =\frac{1}{|c'(t)|}\nabla_t$ respectively denote integration and covariant derivation according to arc length. In the following section, we will show simulations for the special case $a=1$ and $b=1/2$ : for this choice of coefficients, the geodesic equations are easily numerically solved \cite{moi17} if we adopt the so-called square root velocity representation \cite{sri}, in which each curve is represented by the pair formed by its starting point and speed vector renormalized by the square root of its norm. Let us characterize the horizontal subspace for $G^{a,b}$, and give the decomposition of a tangent vector.
\begin{proposition}[Horizontal part of a vector for an elastic metric]\label{horvect}
Let $c\in \mathcal M$ be a smooth immersion. A tangent vector $h \in T_c\mathcal M$ is horizontal for the elastic metric \eqref{metric} if and only if it verifies the ordinary differential equation
\begin{equation}
\label{odehor}
\left((a/b)^2-1\right)\langle \nabla_th, \nabla_tv\rangle - \langle \nabla_t^2h, v \rangle + |c'|^{-1} \langle \nabla_tc', v \rangle\langle \nabla_th, v\rangle=0.
\end{equation}
The vertical and horizontal parts of a tangent vector $w\in T_c\mathcal M$ are given by
\begin{equation*}
w^{ver} = mv, \quad w^{hor} = w - mv,
\end{equation*}
where the real function $m\in C^\infty([0,1],\mathbb R)$ verifies $m(0)=m(1)=0$ and
\begin{equation}
\label{mtt}
\begin{aligned}
&m'' - \langle \nabla_tc'/|c'|, v \rangle m' - (a/b)^2 |\nabla_tv|^2 m \\
&\hspace{3em}= \langle \nabla_t\nabla_tw, v \rangle - \left((a/b)^2-1\right)\langle \nabla_tw, \nabla_tv\rangle -\langle \nabla_tc'/|c'|, v \rangle\langle \nabla_tw, v\rangle.
\end{aligned}
\end{equation}
\end{proposition}
\begin{proof}
Let $h \in T_c\mathcal M$ be a tangent vector. It is horizontal if and only if it is orthogonal to any vertical vector, that is any vector of the form $mv$ with $m\in C^\infty([0,1],\mathbb R)$ such that $m(0)=m(1)=0$. We have $\nabla_t(mv) = m'v+m\nabla_tv \,$ and since $\langle \nabla_tv, v\rangle =0$ we get $\nabla_t(mv)^N = m \nabla_tv$ and $\nabla_t(mv)^T = m'v$. Since $m(0)=0$ the non integral part vanishes and the scalar product is written
\begin{align*}
&G^{a,b}_c(h,mv) = \int_0^1 \left( a^2 m \langle \nabla_th, \nabla_tv \rangle + b^2 m'\langle\nabla_th, v \rangle \right) |c'|^{-1}\mathrm dt\\
&=\int_0^1 a^2 m \langle \nabla_th, \nabla_tv \rangle |c'|^{-1} \mathrm dt - \int_0^1 b^2 m \frac{d}{dt}\big(\langle \nabla_th, v\rangle |c'|^{-1}\big) \mathrm dt\\
&=\int_0^1 m/|c'|\Big( (a^2-b^2)\langle \nabla_th, \nabla_tv\rangle - b^2 \langle \nabla_t\nabla_th, v \rangle+ b^2\langle \nabla_tc', v \rangle\langle \nabla_th, v\rangle|c'|^{-1} \Big) \mathrm dt,
\end{align*}
where we used integration by parts. The vector $h$ is horizontal if and only if $G^{a,b}_c(h,mv)=0$ for all such $m$, and so we obtain the desired equation. Now consider a tangent vector $w$ and a real function $m : [0,1] \rightarrow \mathbb R$ such that $m(0)=m(1)=0$. Then $w-mv$ is horizontal if and only if it verifies the ODE \eqref{odehor}. Noticing that
$\langle \nabla_tv,v\rangle=0$, $\langle \nabla_t\nabla_tv,v\rangle=-|\nabla_tv|^2$ and $\nabla_t\nabla_t(mv)= m'' v+ 2m'\nabla_tv +m\nabla_t\nabla_tv$, we easily get the desired equation.
\end{proof}
This allows us to characterize the horizontal part of a path of curves for $G^{a,b}$.
\begin{proposition}[Horizontal part of a path for an elastic metric]\label{horpathel}
Let $s\mapsto c(s)$ be a path in $\mathcal M$. Then its horizontal part is given by $c^{hor}(s,t) = c(s,\varphi(s)^{-1}(t))$, where the path of diffeomorphisms $s\mapsto \varphi(s)$ is solution of the PDE
\begin{equation}
\label{phicont}
\varphi_s(s,t) = m(s,t)/|c_t(s,t)|\cdot \varphi_t(s,t),
\end{equation}
with initial condition $\varphi(0,\cdot)=\text{Id}$, and where $m(s) : [0,1] \rightarrow \mathbb R$, $t \mapsto m(s,t)$ is solution for all $s$ of the ODE
\begin{equation}\label{odem}
\begin{aligned}
&m_{tt} - \langle \nabla_tc_t/|c_t|, v \rangle m_t - (a/b)^2 |\nabla_tv|^2 m\\
&\hspace{2em}= \langle \nabla_t\nabla_tc_s, v \rangle - \left((a/b)^2-1\right)\langle \nabla_tc_s, \nabla_tv\rangle -\langle \nabla_tc_t/|c_t|, v \rangle\langle \nabla_tc_s, v\rangle.
\end{aligned}
\end{equation}
\end{proposition}
\begin{proof}
This is a direct consequence of Propositions \ref{horpath} and \ref{horvect}.
\end{proof}
We numerically solve the PDE of Proposition \ref{horpathel} using Algorithm \ref{alg:horpath}.
\begin{algorithm}
\KwData{path of curves $s\mapsto c(s)$}
\KwResult{path of diffeomorphisms $s\mapsto \varphi(s)$}
\For{$k=1$ \KwTo $n$}{
estimate the derivative $\varphi_t(\frac{k}{n},\cdot)$\;
solve ODE \eqref{odem} using a finite difference method to obtain $m(\frac{k}{n},\cdot)$\;
set $\varphi_s(\frac{k}{n},t) \leftarrow m(\frac{k}{n},t)/|c_t(\frac{k}{n},t)|\cdot \varphi_t(\frac{k}{n},t)$ for all $t$\;
propagate $\varphi(\frac{k+1}{n},t) \leftarrow \varphi(\frac{k}{n},t) + \frac{1}{n} \varphi_s(\frac{k}{n},t)$ for all $t$\;
}
\caption{Decomposition of a path of curves.}
\label{alg:horpath}
\end{algorithm}

\section{Simulations}	

We test the optimal matching algorithm for the elastic metric with parameters $a=2b=1$ - for which all the formulas and tools to compute geodesics are available \cite{moi17} - and for curves in the plane, the hyperbolic half-plane $\mathbb H^2$ and the sphere $\mathbb S^2$. The curves are discretized and geodesics are computed using a discrete geodesic shooting method presented in detail in \cite{moi17}. Useful formulas and algorithms in $\mathbb H^2$ and $\mathbb S^2$ are available in \cite{moi17} and \cite{zhang16} respectively. Figure \ref{fig:horgeodshootH2} shows results of the optimal matching algorithm for a pair of segments in $\mathbb H^2$. We consider 5 different combinations of parameterizations of the two curves, always fixing the parameterization of the curve on the left-hand side while searching for the optimal reparameterization of the curve on the right-hand side. On the top row, the points are "evenly distributed" along the latter, and on the bottom row, along the former. For each set of parameterizations, the geodesic between the initial parameterized curves (more precisely, the trajectories taken by each point) is shown in blue, and the horizontal geodesic obtained as output of the optimal matching algorithm is shown in red. The two images on the bottom right corner show their superpositions, and their lengths are displayed in Table \ref{fig:table}, in the same order as the corresponding images of Figure \ref{fig:horgeodshootH2}. We can see that the horizontal geodesics redistribute the points along the right-hand side curve in a way that seems natural : similarly to the distribution of the points on the left curve. Their superposition shows that the underlying shapes of the horizontal geodesics are very similar, which is not the case of the initial geodesics. 
The horizontal geodesics are always shorter than the initial geodesics, as expected, and have always approximatively the same length. This common length is the distance between the shapes of the two curves. The same exercise can be carried out on spherical curves (Figure \ref{fig:horgeodshootS2}) and on plane curves, for which we show the superposition of the geodesics and horizontal geodesics between different parameterizations in Figure \ref{fig:superpositionR2}. The execution time varies from a few seconds to a few minutes, depending on the curves and the ambient space : the geodesics between plane curves are computed using explicit equations whereas for curves in a nonlinear manifold, we use a time-consuming geodesic shooting algorithm.

\begin{figure}
\centering
\subfloat{\includegraphics[width=0.18\textwidth]{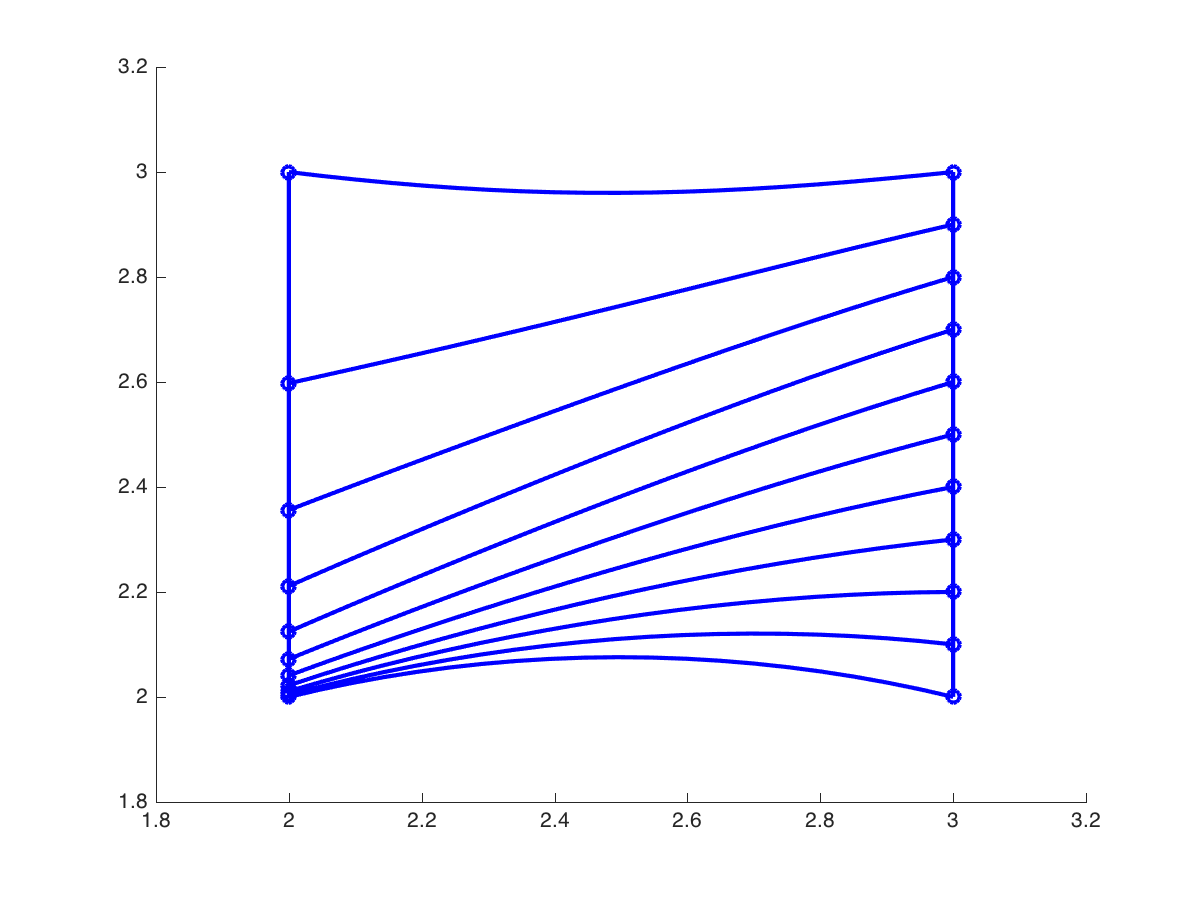}}\hspace*{-1.2em}
\subfloat{\includegraphics[width=0.18\textwidth]{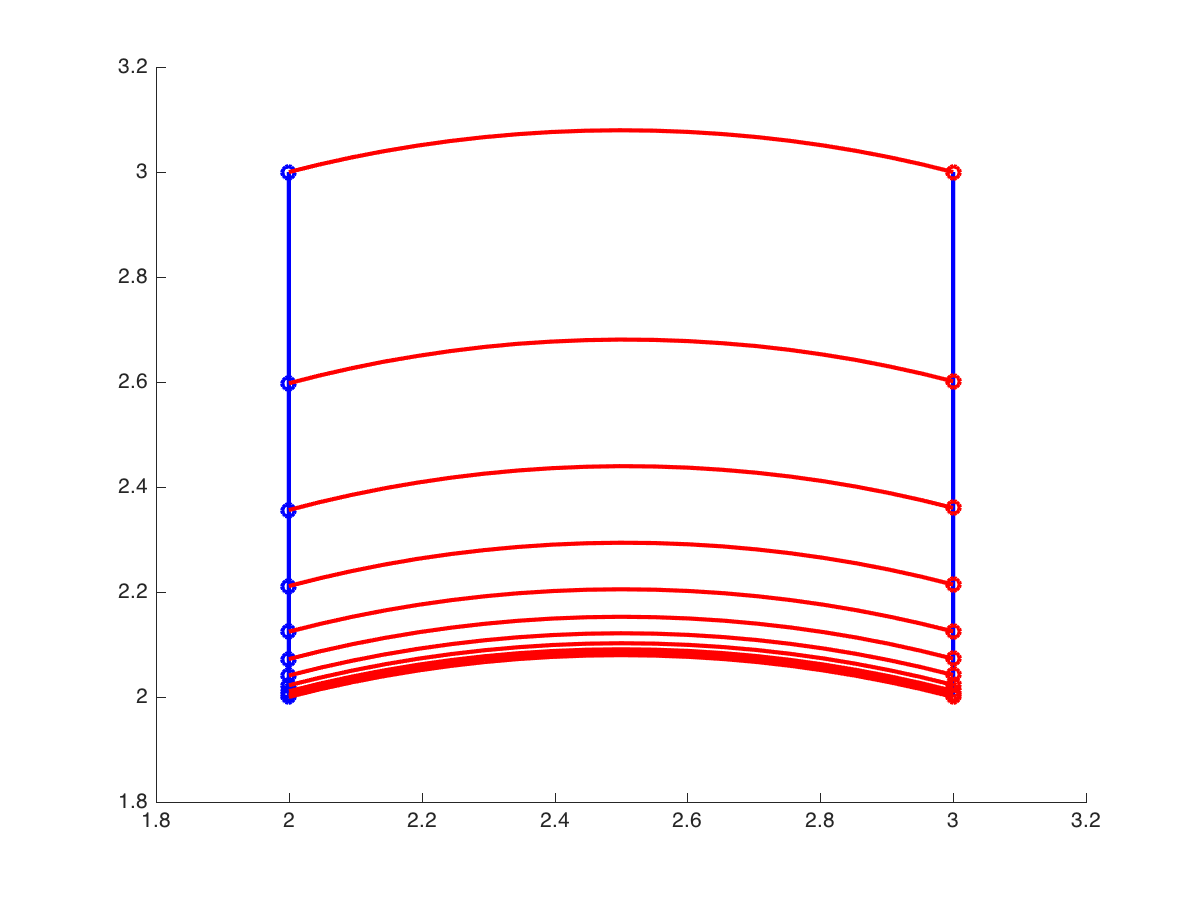}}
\subfloat{\includegraphics[width=0.18\textwidth]{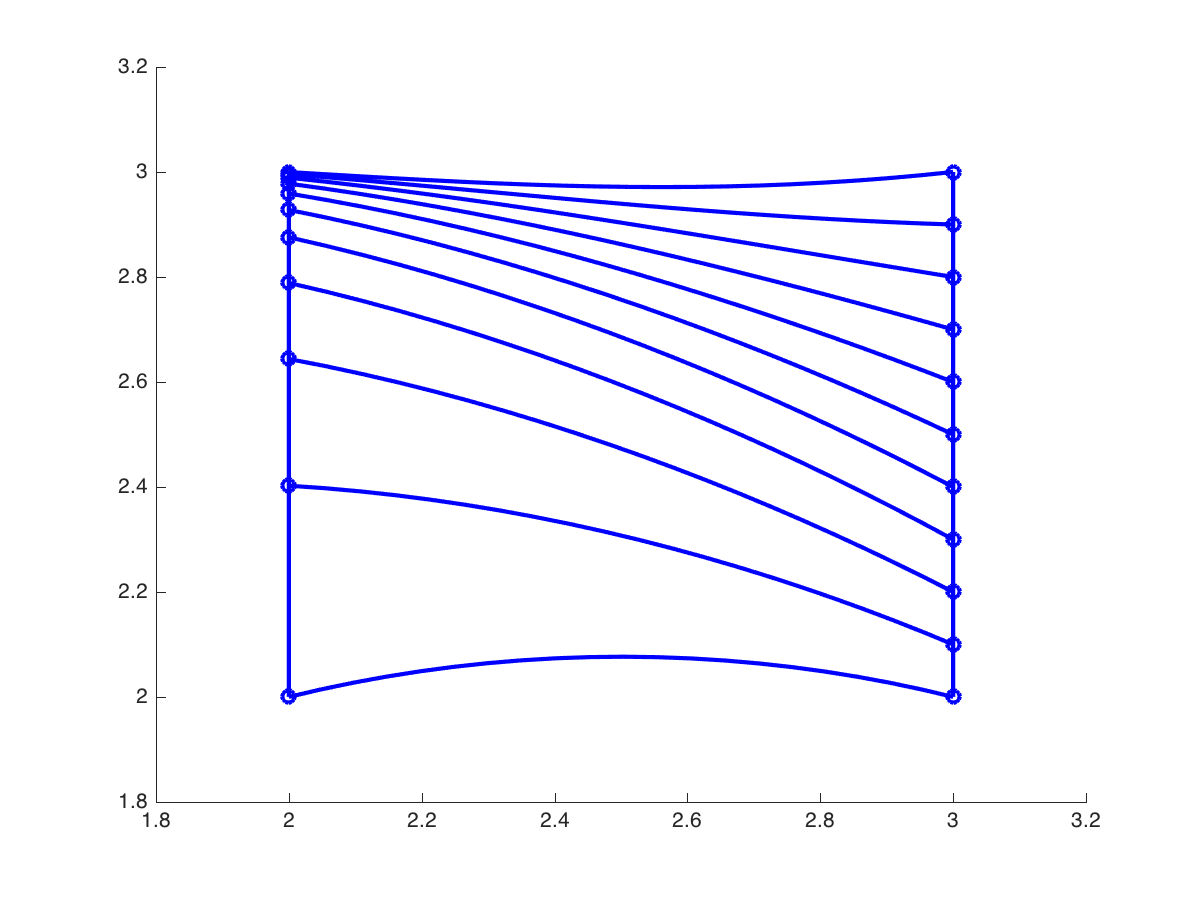}}\hspace*{-1.2em}
\subfloat{\includegraphics[width=0.18\textwidth]{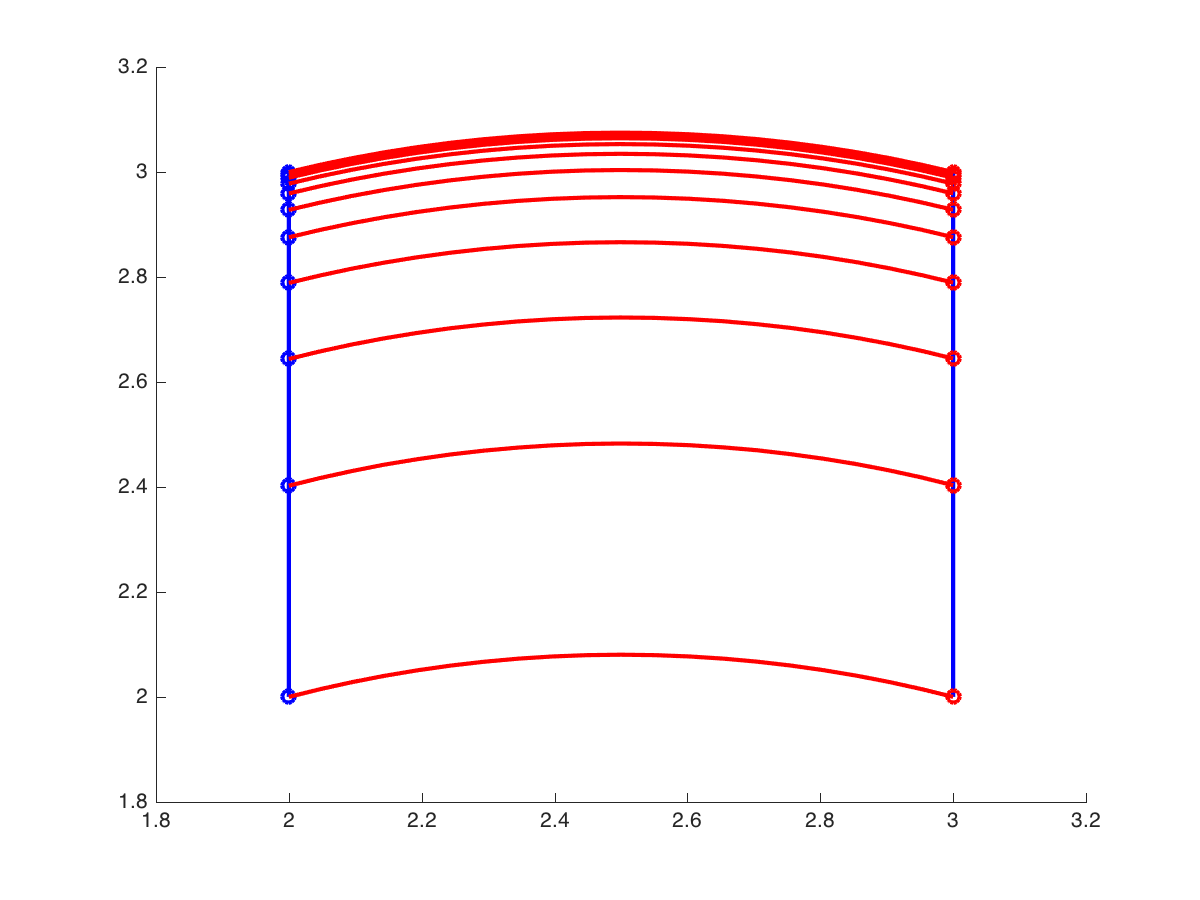}}
\subfloat{\includegraphics[width=0.18\textwidth]{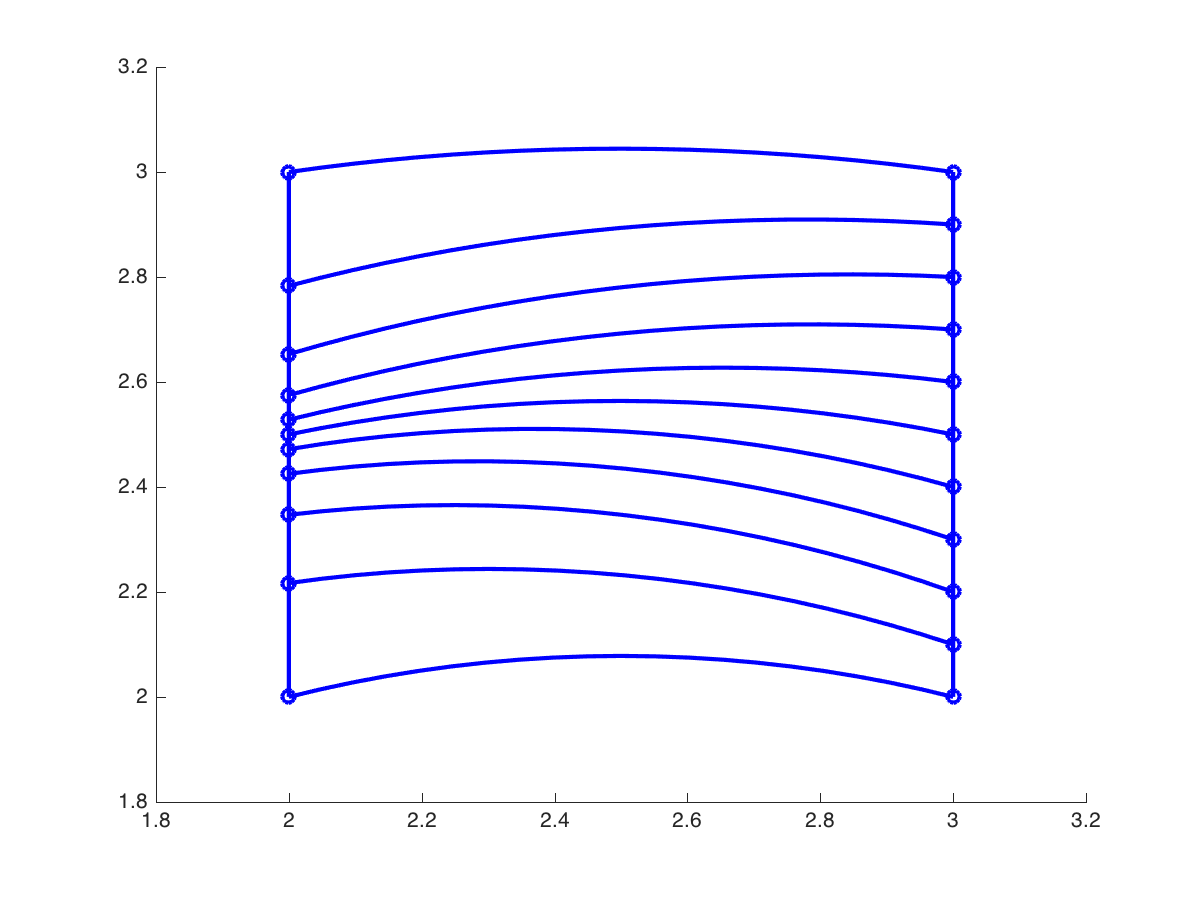}}\hspace*{-1.2em}
\subfloat{\includegraphics[width=0.18\textwidth]{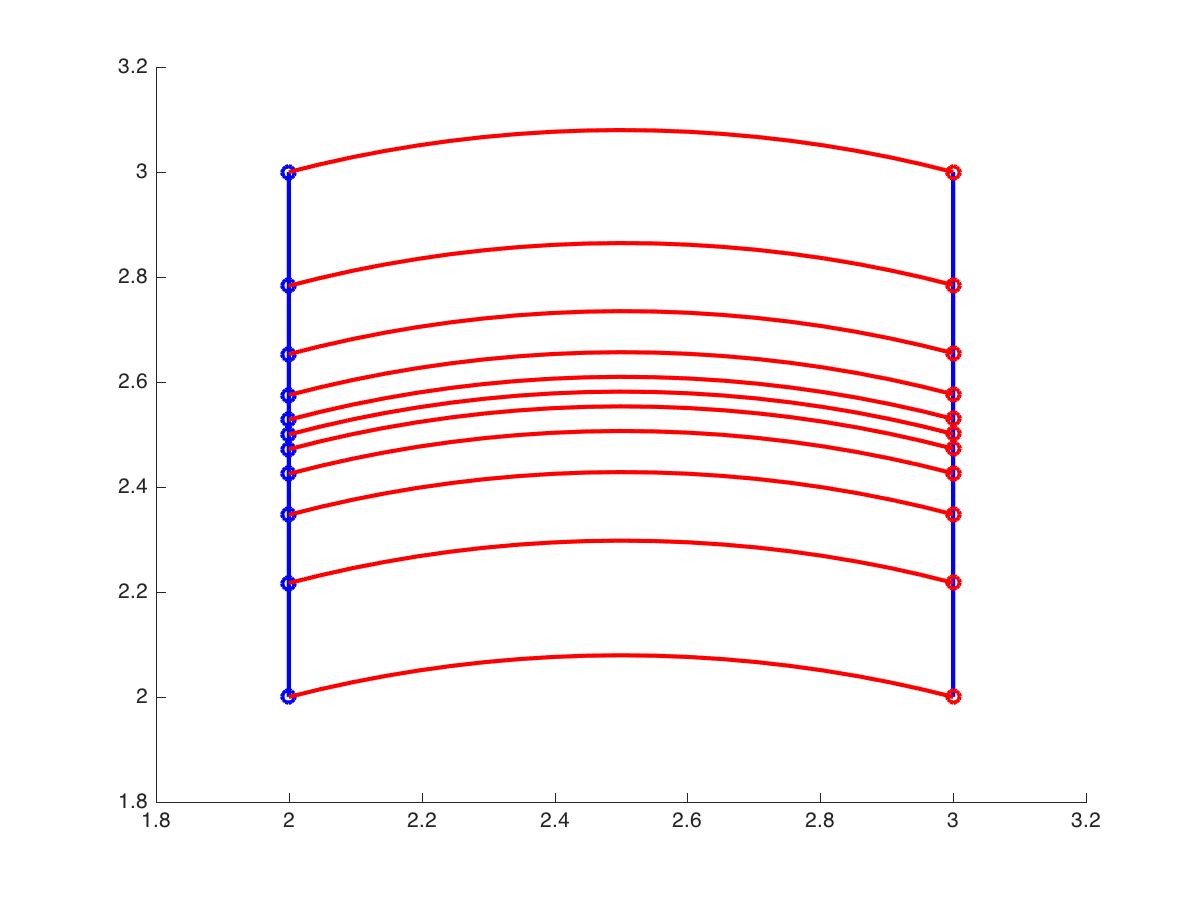}}\vspace*{-1em}\\
\subfloat{\includegraphics[width=0.18\textwidth]{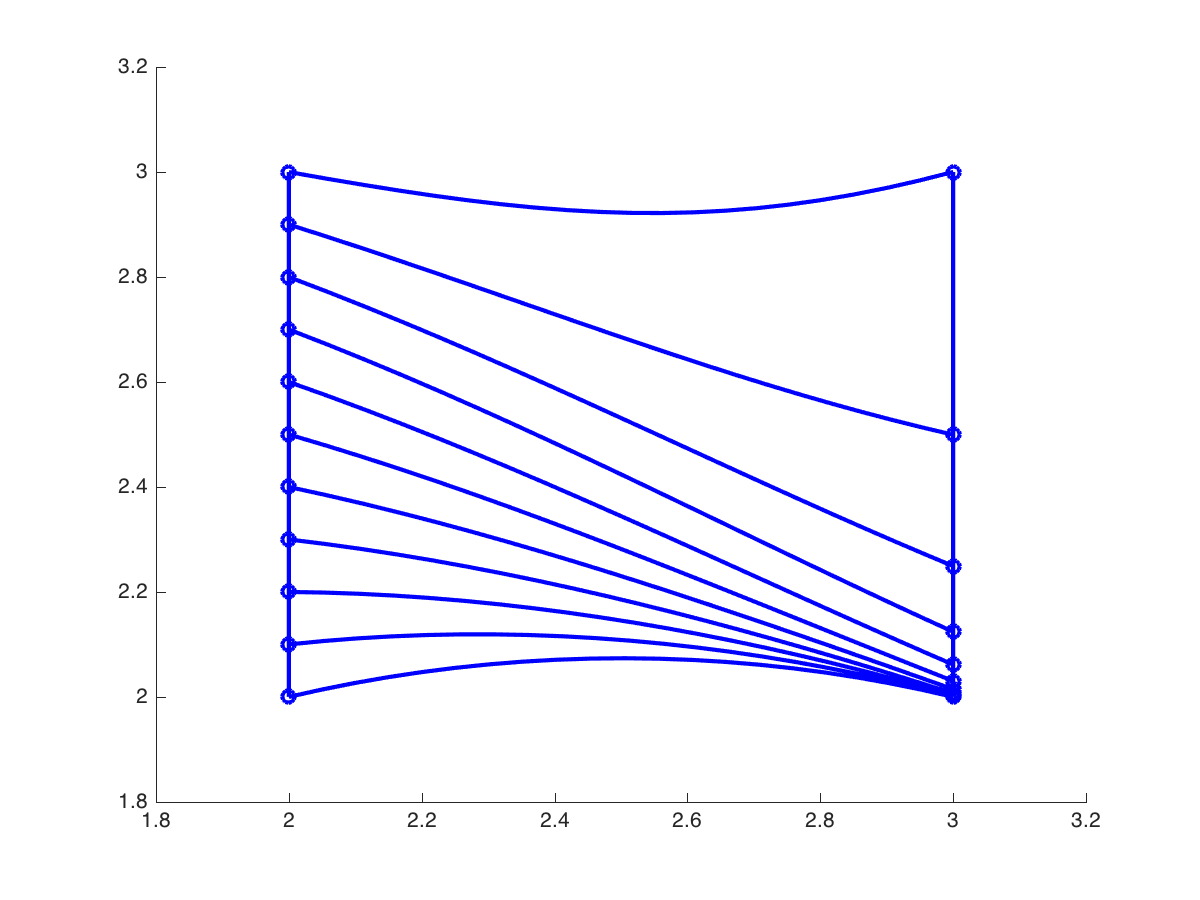}}\hspace*{-1.2em}
\subfloat{\includegraphics[width=0.18\textwidth]{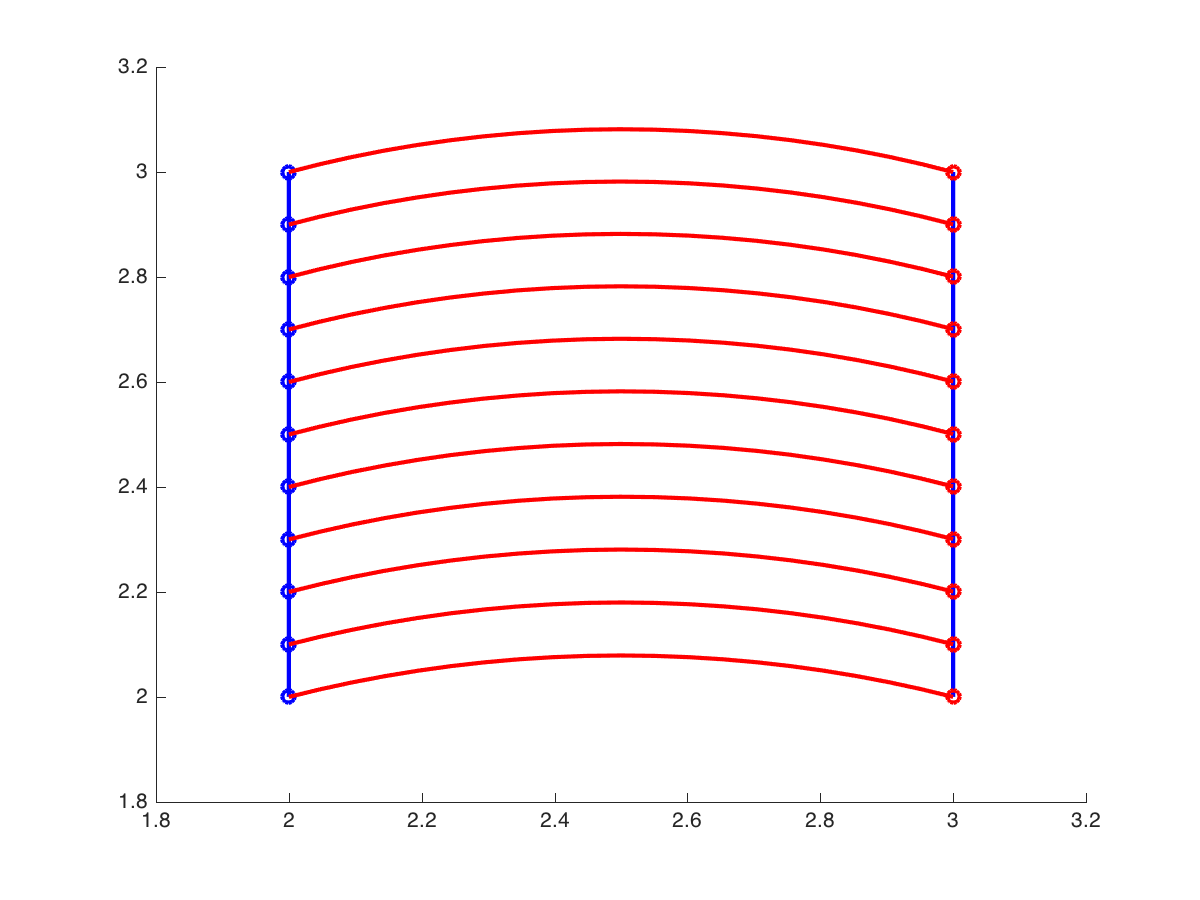}}
\subfloat{\includegraphics[width=0.18\textwidth]{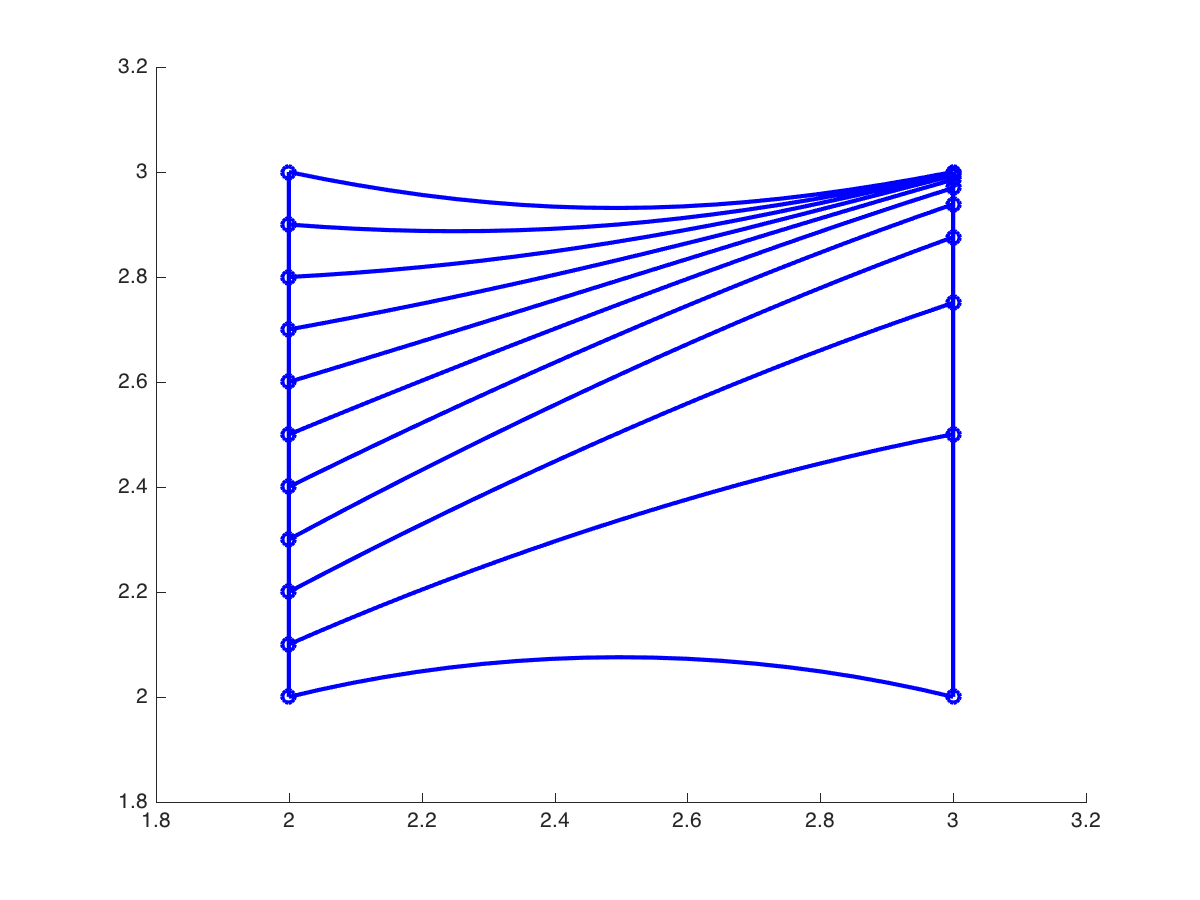}}\hspace*{-1.2em}
\subfloat{\includegraphics[width=0.18\textwidth]{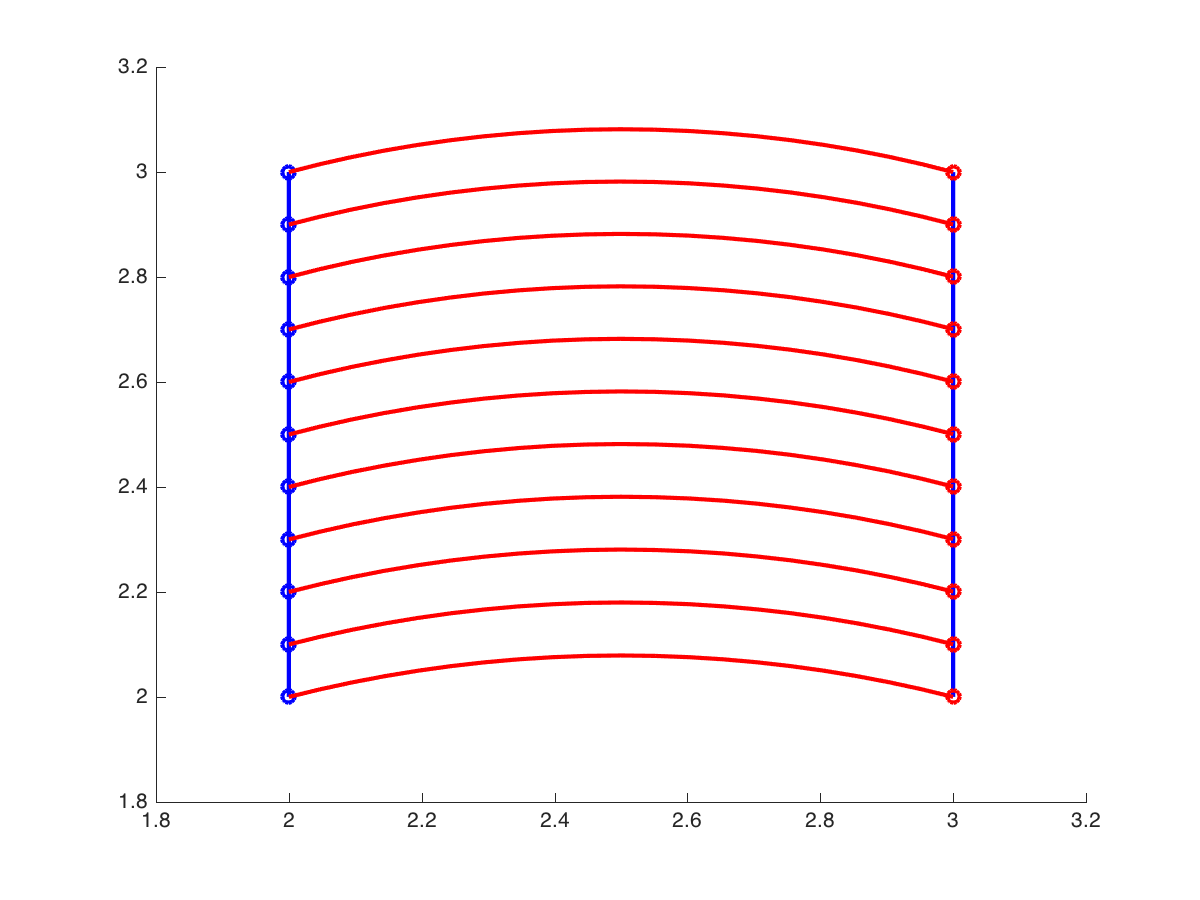}}
\subfloat{\includegraphics[width=0.18\textwidth]{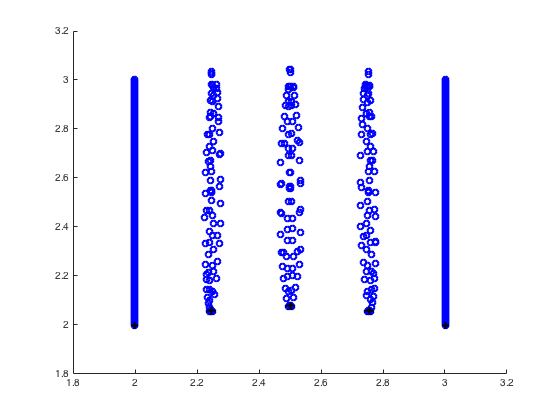}}\hspace*{-1.2em}
\subfloat{\includegraphics[width=0.18\textwidth]{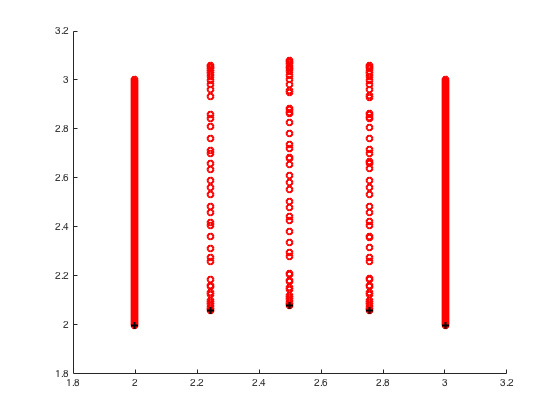}}
\caption{\small Geodesics between parameterized curves (blue) and corresponding horizontal geodesics (red) in the hyperbolic half-plane, and their superpositions.}
\label{fig:horgeodshootH2}
\end{figure}
\begin{figure}
\centering
\caption{\small Length of the geodesics of the hyperbolic half-plane shown in Figure \ref{fig:horgeodshootH2}.}
\label{fig:table}
\begin{tabular}{|c|c|c|c|c|c|}
\hline
\color{blue} 0.6287 & \color{red} 0.5611 & \color{blue} 0.6249 & \color{red} 0.5633 &  \color{blue} 0.5798 & \color{red} 0.5608 \\
\hline
\color{blue} 0.7161 & \color{red} 0.5601 & \color{blue} 0.7051 & \color{red} 0.5601 &\cellcolor{gray!25} & \cellcolor{gray!25} \\
\hline
\end{tabular}
\end{figure}
\begin{figure}
\centering
\subfloat{\includegraphics[width=0.18\textwidth]{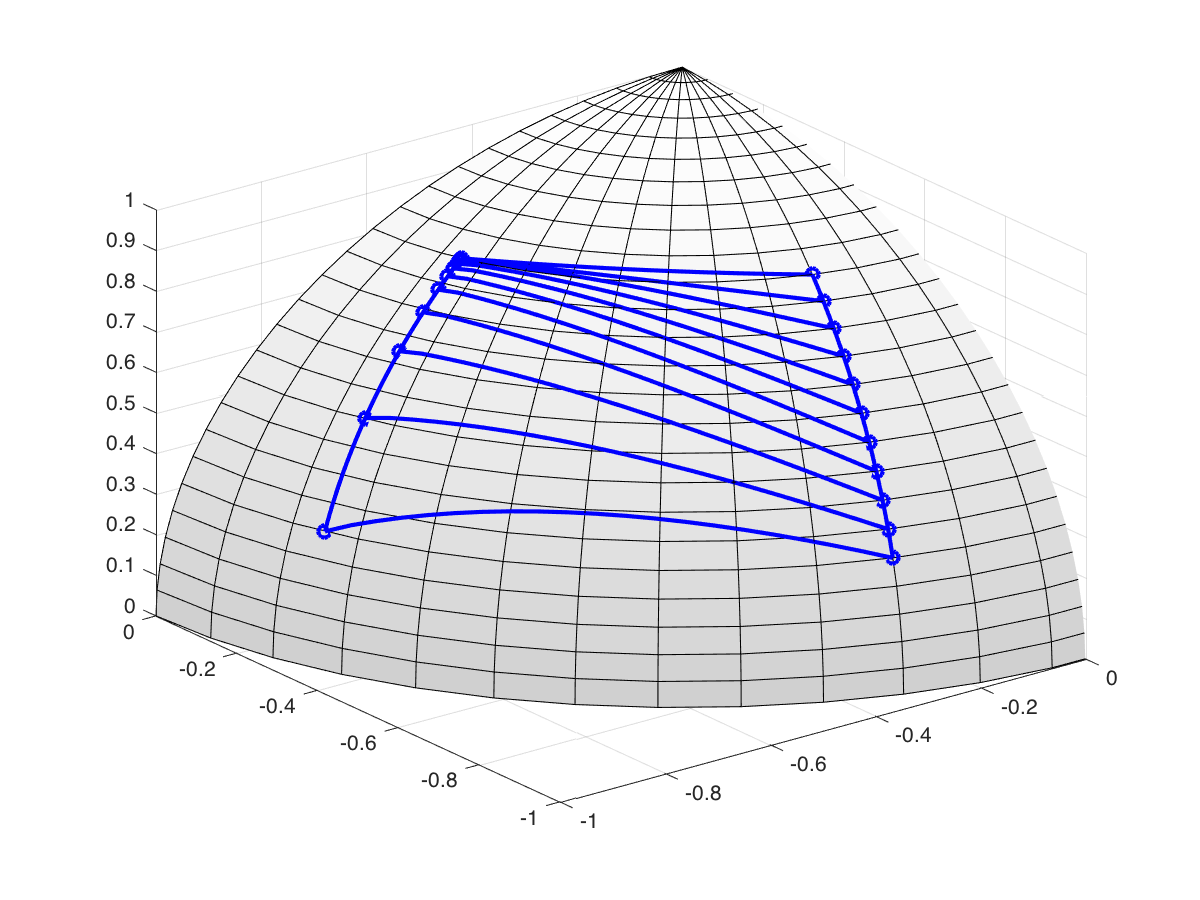}}\hspace*{-1em}
\subfloat{\includegraphics[width=0.18\textwidth]{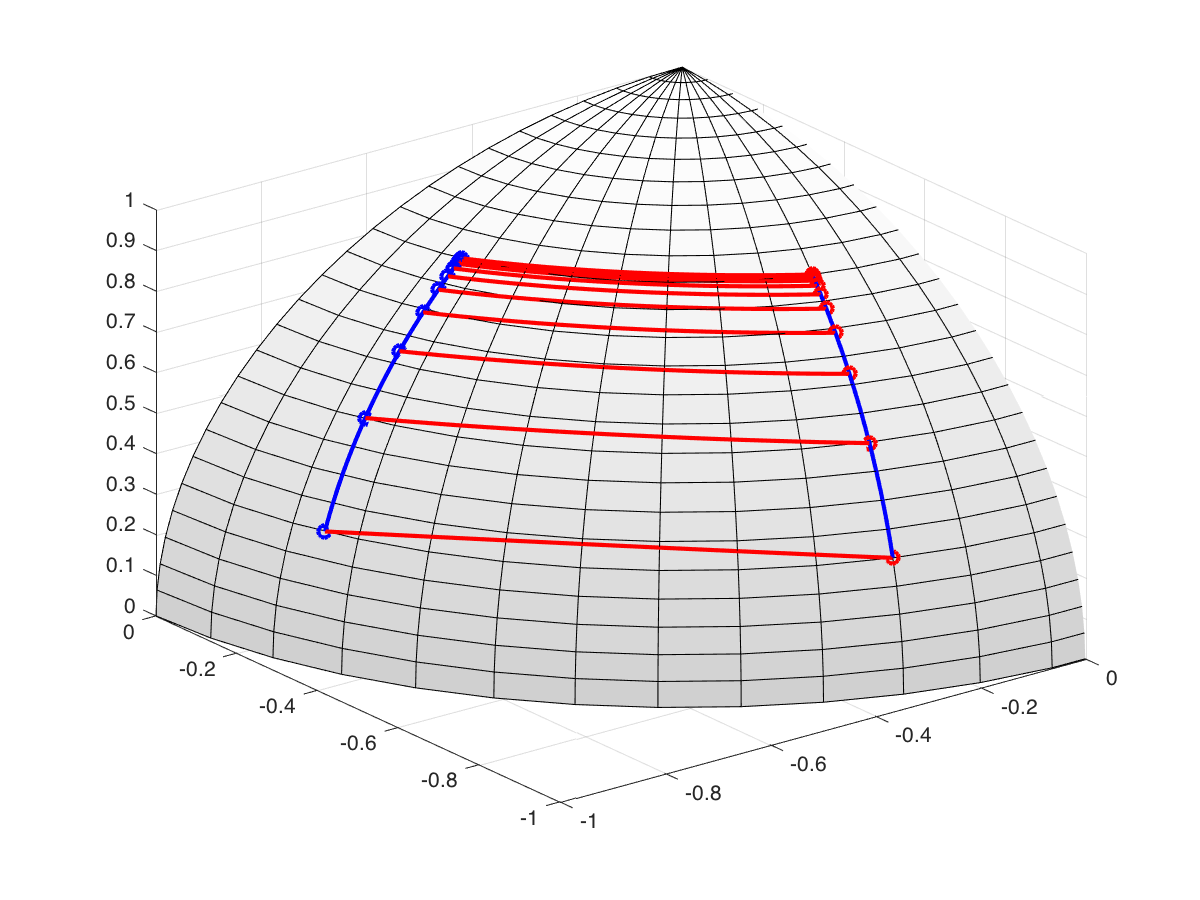}}\hspace*{-0.2em}
\subfloat{\includegraphics[width=0.18\textwidth]{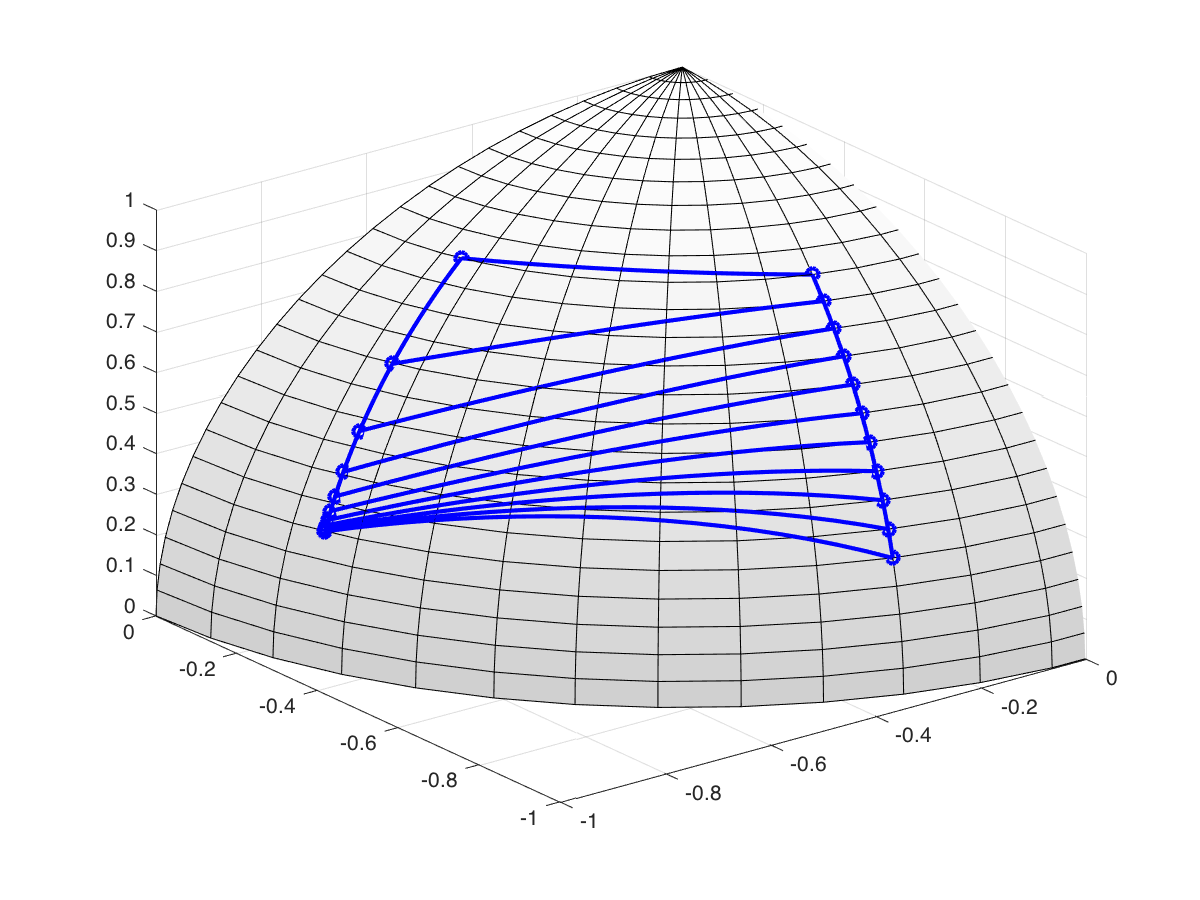}}\hspace*{-1em}
\subfloat{\includegraphics[width=0.18\textwidth]{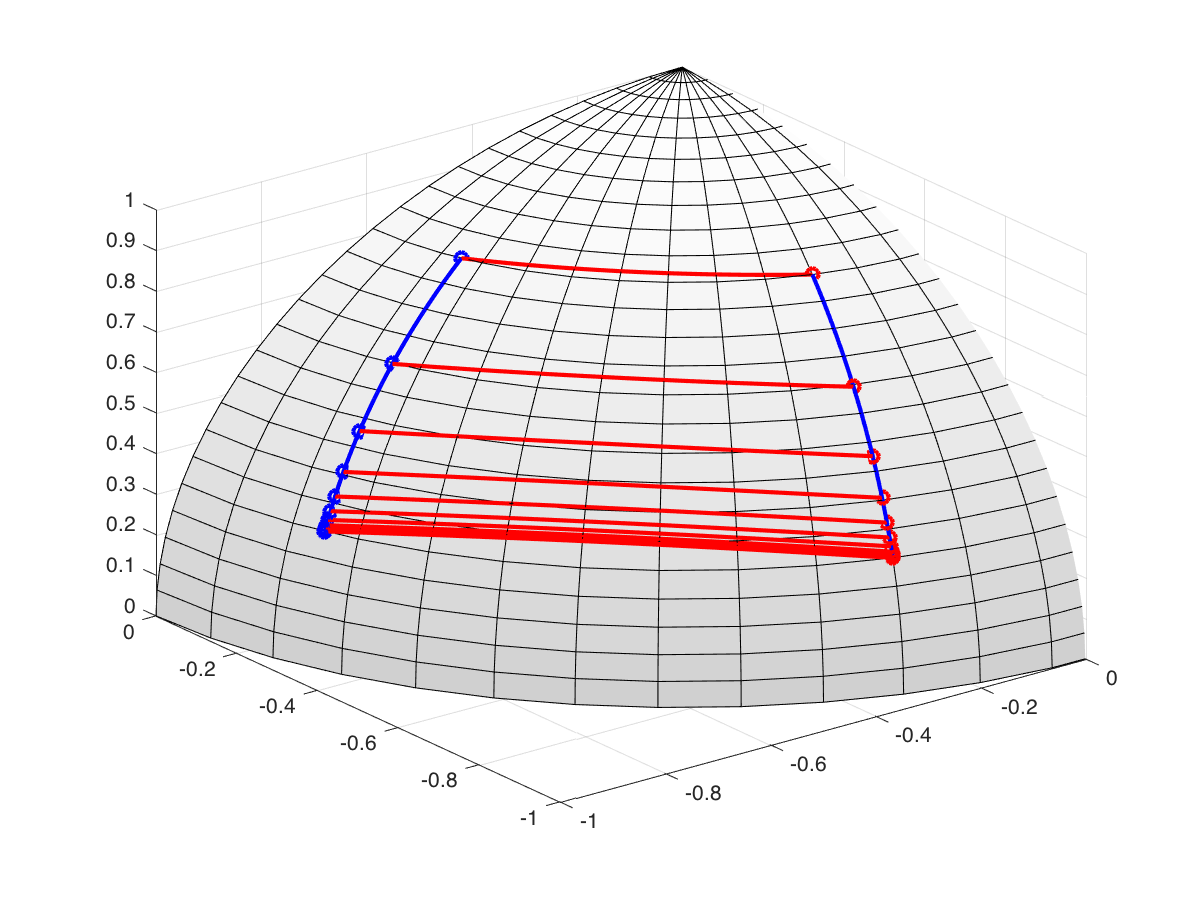}}\hspace*{-0.2em}
\subfloat{\includegraphics[width=0.18\textwidth]{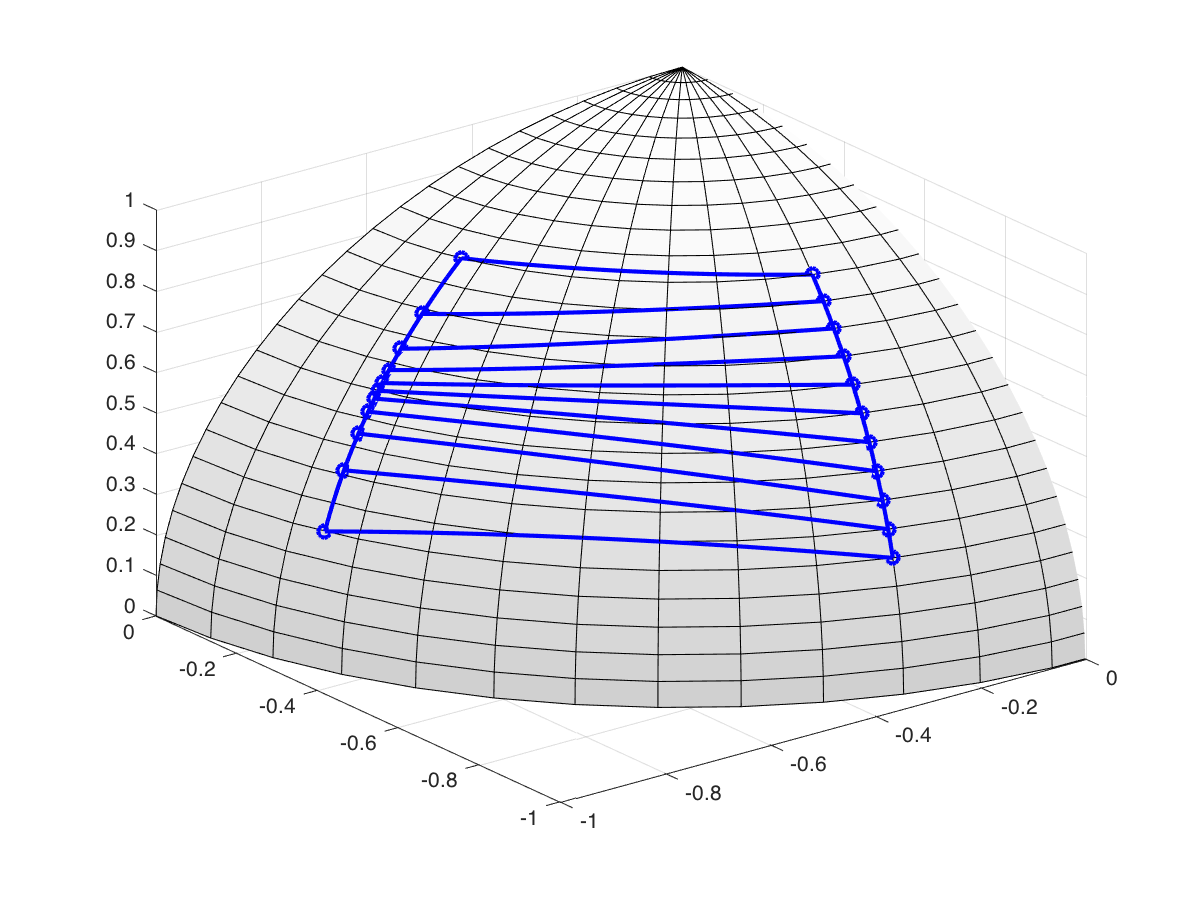}}\hspace*{-1em}
\subfloat{\includegraphics[width=0.18\textwidth]{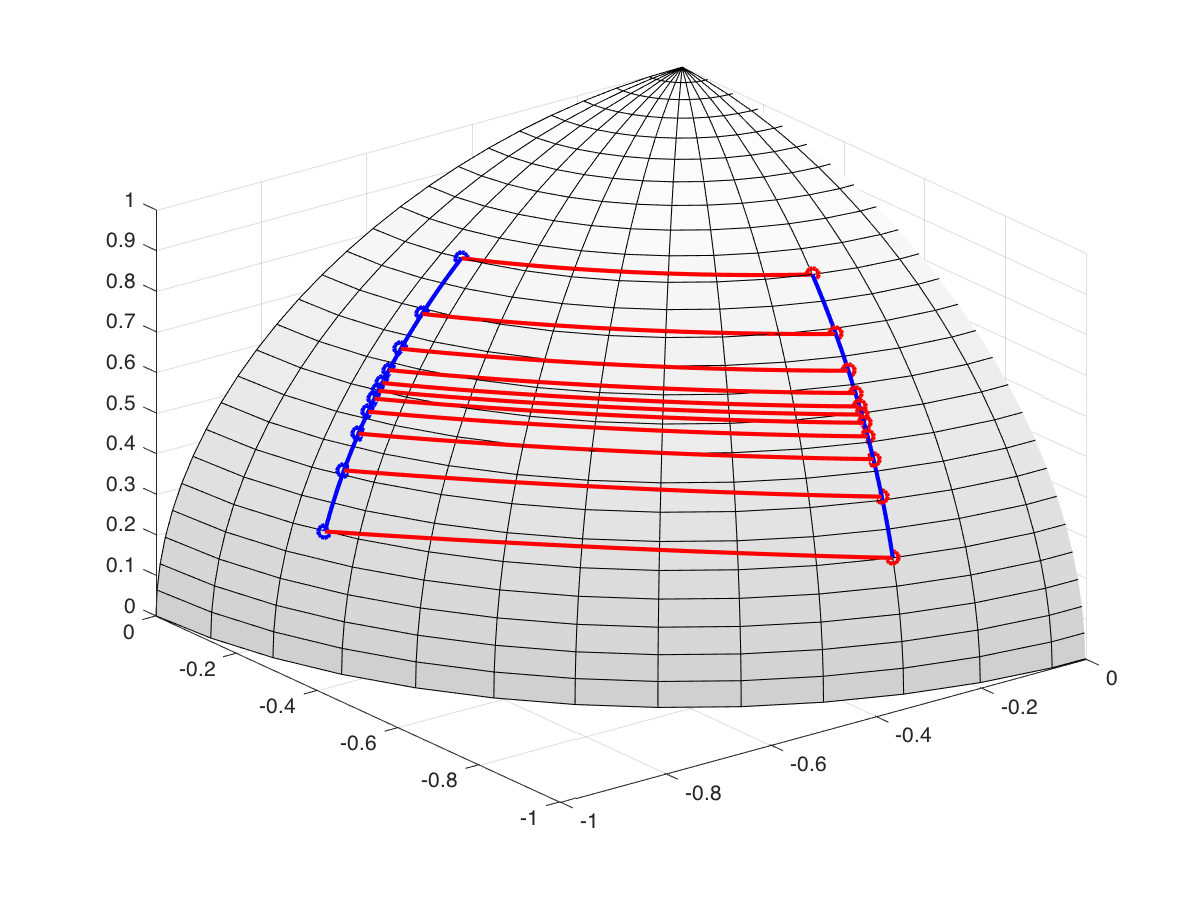}}\vspace*{-1.2em}\\
\subfloat{\includegraphics[width=0.18\textwidth]{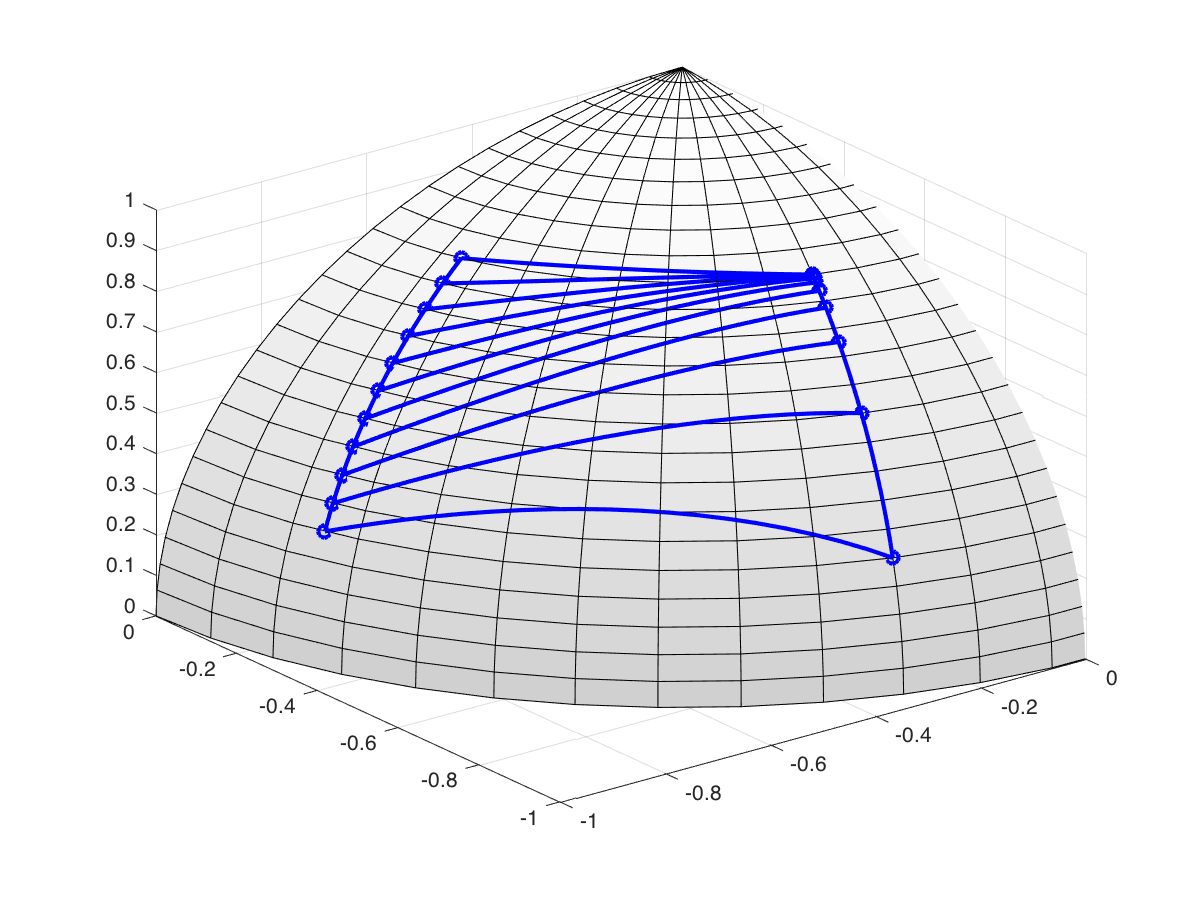}}\hspace*{-1em}
\subfloat{\includegraphics[width=0.18\textwidth]{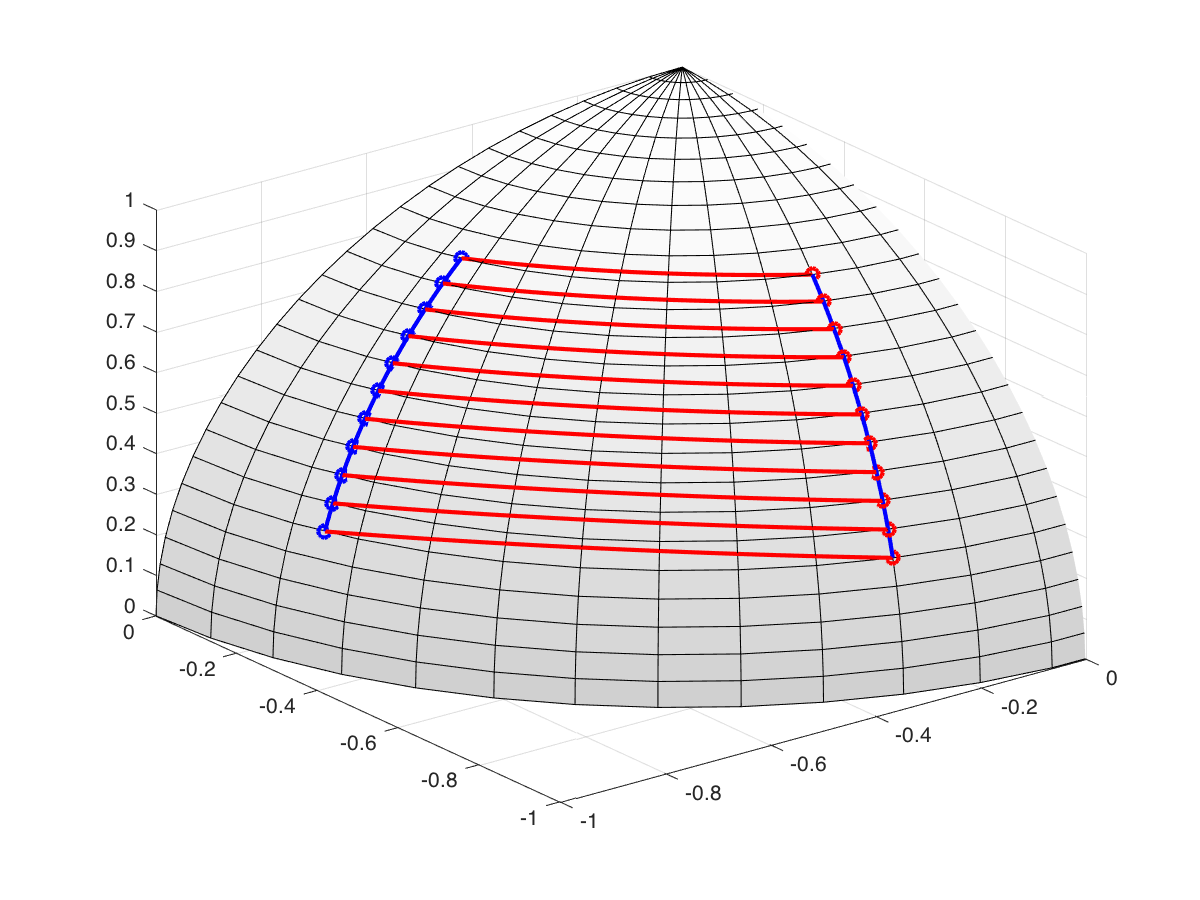}}\hspace*{-0.2em}
\subfloat{\includegraphics[width=0.18\textwidth]{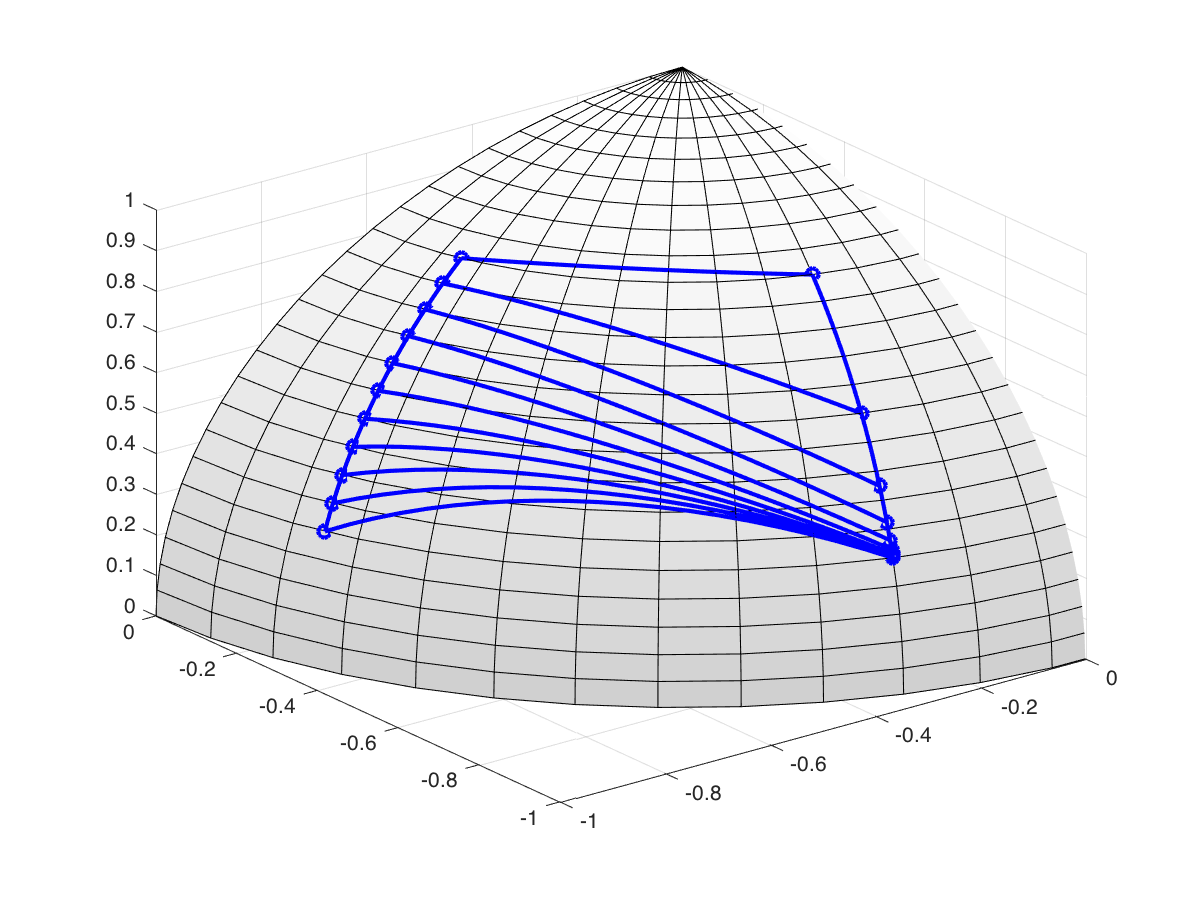}}\hspace*{-1em}
\subfloat{\includegraphics[width=0.18\textwidth]{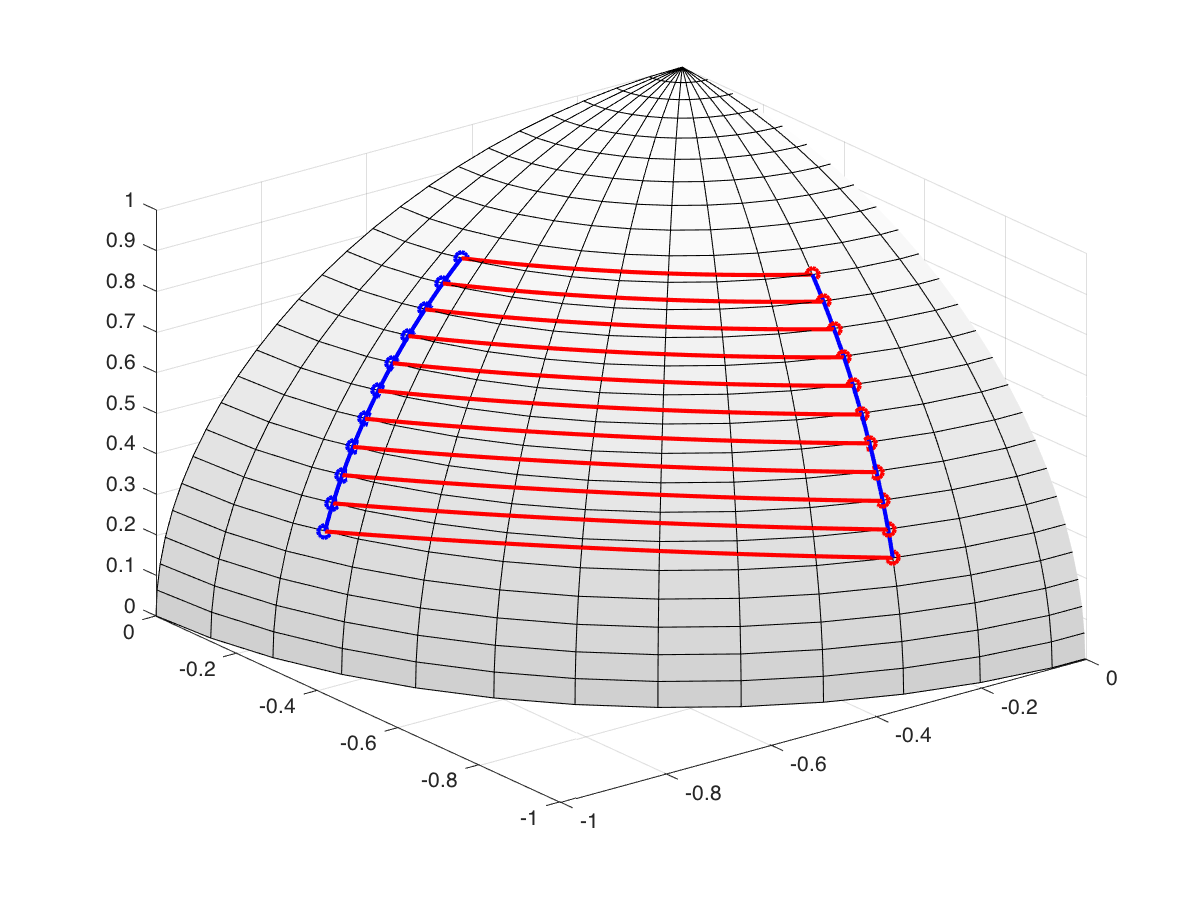}}\hspace*{-0.2em}
\subfloat{\includegraphics[width=0.18\textwidth]{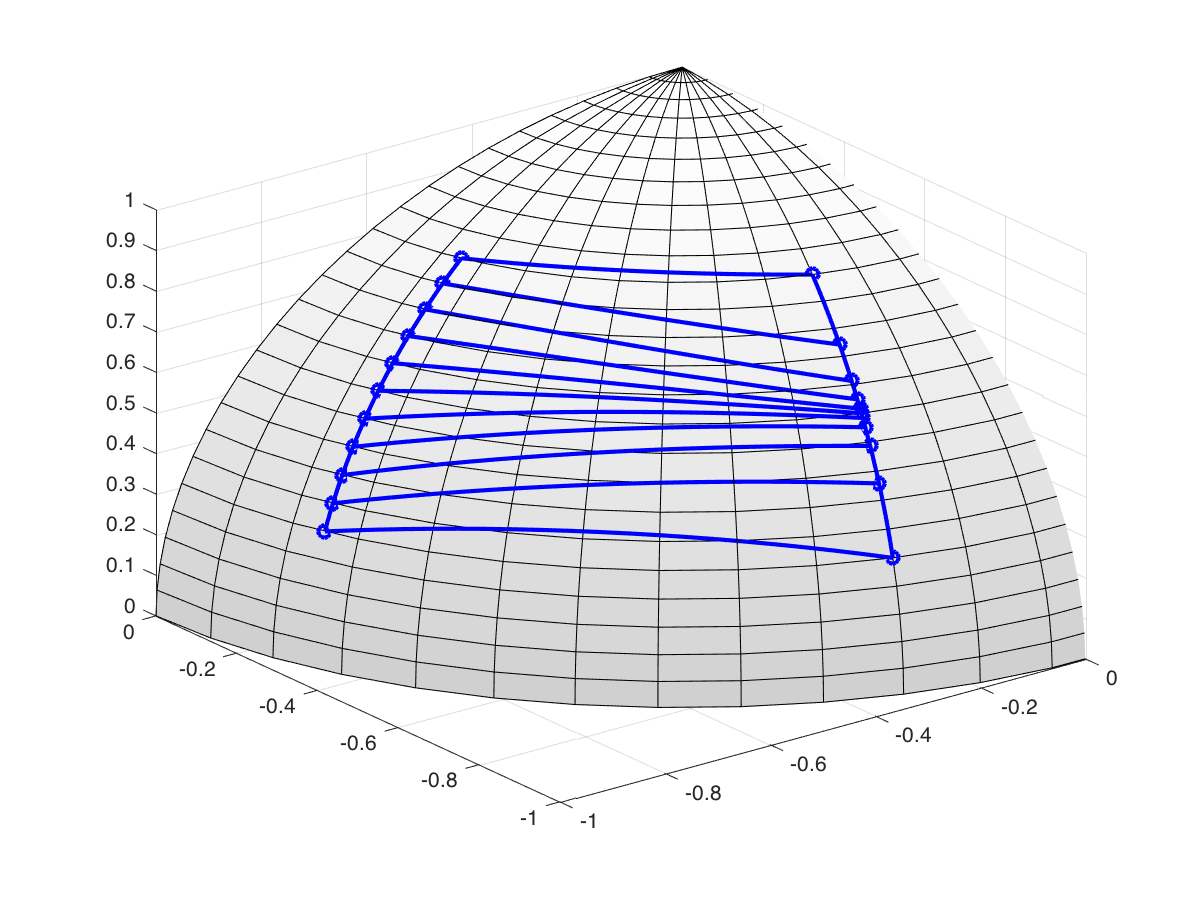}}\hspace*{-1em}
\subfloat{\includegraphics[width=0.18\textwidth]{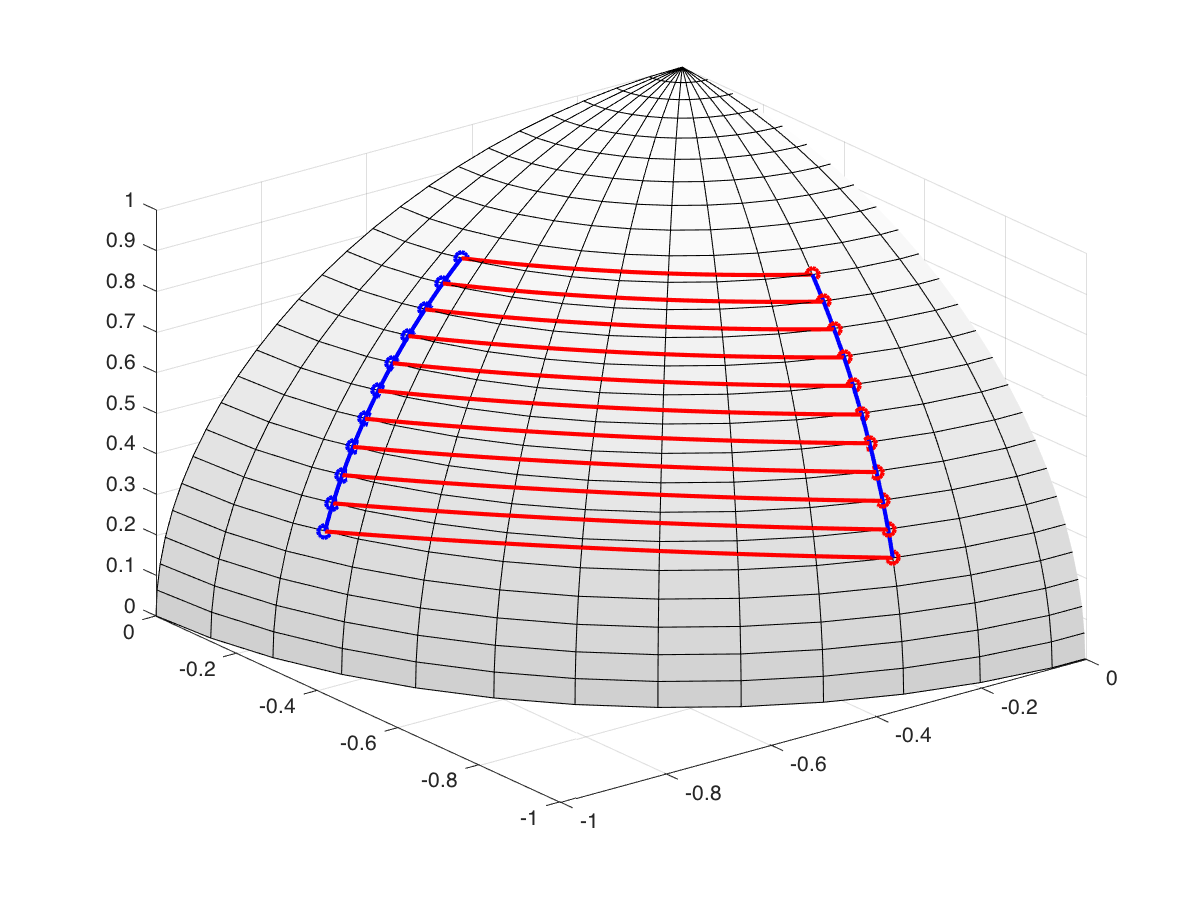}}
\caption{\small Initial and horizontal geodesics between spherical parameterized curves.}
\label{fig:horgeodshootS2}
\end{figure}
\begin{figure}
\centering
\subfloat{\includegraphics[width=0.17\textwidth]{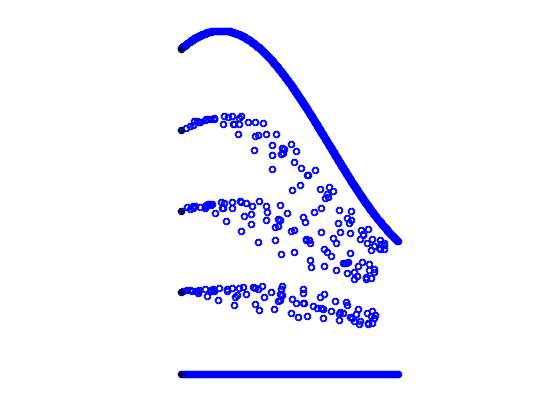}}\hspace*{-2em}
\subfloat{\includegraphics[width=0.17\textwidth]{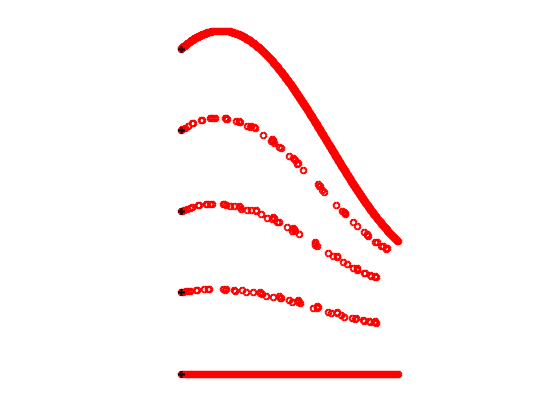}}
\subfloat{\includegraphics[width=0.17\textwidth]{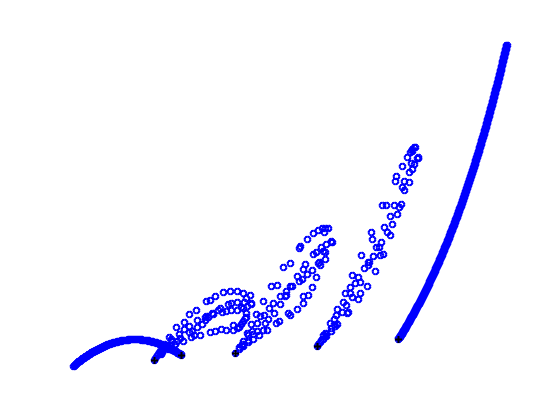}}\hspace*{-1em}
\subfloat{\includegraphics[width=0.17\textwidth]{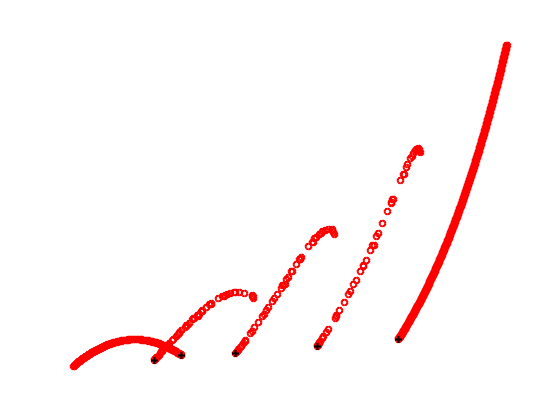}} \hspace{1em}
\subfloat{\includegraphics[width=0.17\textwidth]{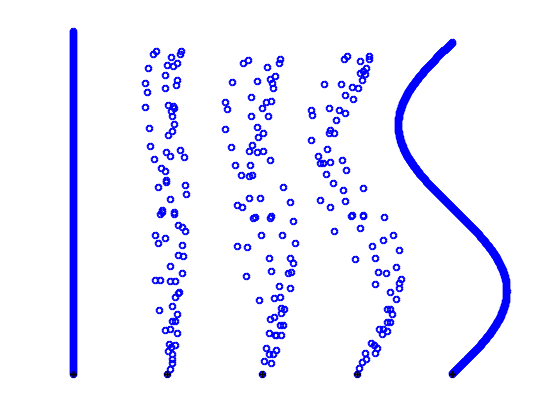}}
\subfloat{\includegraphics[width=0.17\textwidth]{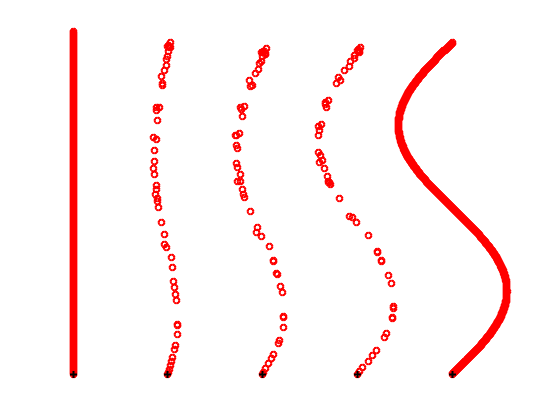}}
\caption{\small Superposition of the initial (blue) and horizontal (red) geodesics obtained for different sets of parameterizations of three pairs of plane curves.}
\label{fig:superpositionR2}
\end{figure}


\end{document}